\documentclass[11pt]{amsart}
\usepackage[margin=1in]{geometry}

\usepackage{amssymb}
\usepackage{amsthm}
\usepackage{amsmath}
\usepackage{mathrsfs}
\usepackage{mathtools}
\usepackage{amsbsy}
\usepackage[all]{xy}
\usepackage{bm}
\usepackage{hyperref}
\usepackage{tikz}
\usepackage{array}
\usepackage{float}
\usepackage{enumerate}
\usepackage{xcolor}
\usepackage[normalem]{ulem}
\usepackage{hhline}
\allowdisplaybreaks
\usepackage[noadjust]{cite}

\usepackage{caption}
\usepackage{subcaption}
\usepackage{tabu}
\usepackage{diagbox}
\usepackage{bbm}
\usepackage{booktabs}

\usepackage[noabbrev,capitalise]{cleveref}

\newenvironment{enumerate*}%
  {\begin{enumerate}[(I)]%
    \setlength{\itemsep}{10pt}%
    \setlength{\parskip}{0pt}}%
  {\end{enumerate}}

\newtheorem{theorem}{Theorem}[section]
\newtheorem{proposition}[theorem]{Proposition}

\newtheorem{lemma}[theorem]{Lemma}

\theoremstyle{definition}

\DeclarePairedDelimiter{\abs}{\lvert}{\rvert}
\DeclarePairedDelimiter{\gen}{\langle}{\rangle}

\newcommand{\cA}{\mathcal A}
\newcommand{\cB}{\mathcal B}

\newcommand{\cP}{\mathcal P}
\newcommand{\cS}{\mathcal S}
\newcommand{\EE}{\mathbb E}
\newcommand{\NN}{\mathbb N}
\newcommand{\ZZ}{\mathbb Z}
\newcommand{\eps}{\varepsilon}
\DeclareMathOperator{\Part}{part}

\title{Statistics of Erd\H{o}s--R\'enyi random numerical semigroups}

\author{Noah Kravitz}
\address{Department of Mathematics, Princeton University, Princeton, NJ 08540, USA}
\email{nkravitz@princeton.edu}

\author{Santiago Morales}
\address{Graduate Group in Applied Mathematics, University of California, Davis, Davis, CA 95616, USA}
\email{moralesduarte@ucdavis.edu}

\author{Carl Schildkraut}
\address{Department of Mathematics, Stanford University, Stanford, CA 94305, USA}
\email{carlsch@stanford.edu}

\begin{document}

\begin{abstract}
For $p>0$ a small parameter, let $\mathcal A \subseteq \mathbb{Z}_{>0}$ be a random subset where each positive integer is included independently with probability $p$.  We show that, with high probability (as $p \to 0$), the numerical semigroup $\langle\mathcal A\rangle:=\{a_1+\cdots+a_k: k \geq 0, a_1, \ldots, a_k \in \mathcal A\}$ generated by $\mathcal A$ has Frobenius number and genus of size $\asymp p^{-1}(\log p^{-1})^2$ and embedding dimension of size $\asymp (\log p^{-1})^2$.  This resolves an open problem of Bogart and the second author.
\end{abstract}

\maketitle

\section{Introduction}
A \emph{numerical semigroup} is a co-finite subset of $\ZZ_{\geq 0}$ which contains $0$ and is closed under addition.  Every numerical semigroup $\cS$ can be expressed (non-uniquely) in terms of generators as
$$\cS=\gen{\cA}:=\{a_1+\cdots+a_k: k \geq 0, a_1, \ldots, a_k \in \cA\},$$
where $\cA \subseteq \ZZ_{>0}$ satisfies $\gcd(\cA)=1$. Among the many invariants attached to a numerical semigroup $\cS$, three of the most classically studied are:
\begin{itemize}
    \item its \emph{Frobenius number} $F(\cS):=\sup (\mathbb{Z}_{> 0} \setminus \cS)$ (the largest gap);
    \item its \emph{genus} $g(\cS):=\abs{\mathbb{Z}_{> 0} \setminus \cS}$ (the number of gaps);
    \item its \emph{embedding dimension} $e(\cS):=\min \{\abs{\cA}: \gen{\cA}=\cS\}$ (the smallest generating set).
\end{itemize}
We use the convention $F(\ZZ_{\geq 0}):=-1$.  These three invariants are closely related and satisfy the elementary inequalities $\frac{F(\cS)+1}{2} \leq g(\cS) \leq F(\cS)+1$ and $e(\cS) \leq F(\cS)+2$ \cite[Lemmas~2~and~3~and~Proposition~5]{Assi2020}.


There has been much work on enumerating numerical semigroups according to these and other invariants, and on determining how these invariants relate to one another typically and extremally~\cite{Arnold1999, Delgado2020, Rosales2009, Assi2020}.  
Another popular topic of study has been the behavior of ``random'' numerical semigroups under various models; see~\cite{Aliev2011, BM,BKS} for further discussion.

In this paper we are concerned with the so-called \emph{Erd\H{o}s--R\'enyi model} of random numerical semigroups, as first introduced by De Loera, O'Neill, and Wilburne~\cite{DLOW} in analogy with Erd\H{o}s--R\'enyi random graphs.\footnote{More precisely, De Loera, O'Neill, and Wilburne studied a \emph{finitary} version of the model where the set $\cA$ is constrained to lie in $[M]$ for some large parameter $M$.  The infinitary version first appeared in the work of Bogart and the second author~\cite{BM}.}  For $p>0$ a (small) parameter, let $\cA \subseteq \mathbb{Z}_{>0}$ be a random subset where each positive integer is included independently with probability $p$, and form the random numerical semigroup $\gen{\cA}$. With probability $1$, the set $\cA$ contains some two consecutive integers; this implies that $\gen{\cA}$ is co-finite and thus a genuine numerical semigroup.

De Loera, O'Neill, and Wilburne~\cite{DLOW} showed that the expected Frobenius number, genus, and embedding dimension can be estimated as
\[p^{-1} \lesssim \EE(g(\gen{\cA})),\ \EE(F(\gen{\cA})) \lesssim p^{-2} \quad \text{and} \quad 3+O(p) \leq \EE(e(\gen{\cA})) \lesssim p^{-1}.\]
Bogart and the second author~\cite{BM} obtained the improved upper bounds
$$\EE(g(\gen{\cA})),\ \EE(F(\gen{\cA})) \lesssim p^{-1}(\log p^{-1})^3 \quad \text{and} \quad \EE(e(\gen{\cA})) \lesssim (\log p^{-1})^3,$$
and they posed the problem of determining the correct polylogarithmic factors. Our main result pins down the true growth rates of these invariants, both in expectation and with high probability  (see \cref{sec:notation} for our probabilistic and asymptotic notation).

\begin{theorem}\label{thm:main}
Let $\cA \subseteq \ZZ_{>0}$ be a density-$p$ random set as described above.  Then with high probability (as $p \to 0$) we have $$g(\gen{\cA}),\ F(\gen{\cA}) \asymp p^{-1} (\log p^{-1})^2 \quad \text{and} \quad e(\gen{\cA}) \asymp (\log p^{-1})^2.$$
Moreover, we have
$$\EE(g(\gen{\cA})),\ \EE(F(\gen{\cA})) \asymp p^{-1} (\log p^{-1})^2 \quad \text{and} \quad \EE(e(\gen{\cA})) \asymp (\log p^{-1})^2.$$
\end{theorem}


The implicit constants in this theorem are quite reasonable.  For instance, in the first with-high-probability statement, one can establish the upper bound $F(\gen{\cA}) \leq (4+o(1))p^{-1} (\log_2 p^{-1})^2$. We remark that since $\frac{F(\cS)+1}2\leq g(\cS)\leq F(\cS)+1$ for any numerical semigroup $\cS$, the results for Frobenius number and genus are transparently equivalent.

The main content of \cref{thm:main} is the with-high-probability statement; from there, the expectation statement merely requires crude control the right tails.  We use counting arguments to establish lower bounds in \cref{sec:lower}, and we use the second moment method to establish upper bounds in \cref{sec:upper}.  We conclude with a few suggestions for future work in \cref{sec:concluding}.

\subsection{Notation and preliminaries}\label{sec:notation}
Throughout the paper, $0<p<1$ is the density parameter for the random set $\cA$.  
We say that an event occurs \emph{with high probability} if its probability of occurring tends to $1$ as $p \to 0$.  We will sometimes assume without comment that $p$ is sufficiently small ($p<1/10$ should suffice).

For asymptotic notation, we write $f \lesssim g$ or $f=O(g)$ if there is a universal constant $C>0$ such that $\abs{f} \leq Cg$, and we write $f \asymp g$ if $f \lesssim g \lesssim f$.  We write $f=o(g)$ if $f/g \to 0$ in the relevant limit (usually as $p \to 0$, but in \cref{thm:partitions} as $n \to \infty$).

We write $[M,N]:=\{M,M+1,\ldots, N\}$ for integers $M \leq N$, and we adopt the shorthand $[N]:=[1,N]=\{1,2,\ldots, N\}$.  All logarithms are base-$e$ unless otherwise specified.  We omit floor and ceiling functions when there is no risk of confusion.

We will make use of the following form of Chernoff's concentration bounds for binomial random variables (see \cite[Theorems~A.1.12~and~A.1.13]{AlonSpencer}).

\begin{theorem}[Chernoff's bound]
Let $0<\alpha<1<\beta$, and let $n$ be a positive integer.  Let $X_1,\ldots,X_n$ be independent Bernoulli-$p$ random variables, and set $X:=X_1+\cdots+X_n$ and $\mu:=pn=\EE X$.  Then
    \[\Pr[X\leq \alpha\mu]\leq\exp\left(-\frac12(1-\alpha)^2\mu\right) \quad \text{and} \quad \Pr[X\geq \beta\mu]\leq\exp\big(-(1-\beta+\beta\log\beta)\mu\big).\]
\end{theorem}

\section{Lower bounds}\label{sec:lower}

We start with the lower bounds in \cref{thm:main}.  The lower bounds for the with-high-probability statements imply the corresponding lower bounds for the expectation statements, so we prove only the former.  In light of the deterministic inequality $g(\cS) \leq F(\cS)+1$, it suffices to show that $$g(\gen{\cA}) \gtrsim p^{-1} (\log p^{-1})^2 \quad \text{and} \quad e(\gen{\cA})\gtrsim (\log p^{-1})^2$$
with high probability.

We have $\EE(\abs{\cA \cap [M]})=pM$ for every positive integer $M$.  The rough idea is that if $\abs{\cA \cap [M]}\lesssim pM$ for all $M$ (which, more or less, occurs with high probability), then we can use classical asymptotics for the integer partition function to show that $\abs{\gen{\cA} \cap [N]}$ is much smaller than $N$, where $N$ is a suitable small constant times $p^{-1}(\log p^{-1})^2$.  This suffices for the lower bound on $g(\gen{\cA})$.  For the lower bound on $e(\gen{\cA})$, we further argue that $\gen{\cA} \cap [N]$ is so sparse that most of the elements of $\cA \cap [N/2,N]$ are not representable as sums of other elements of $\cA$.


\subsection{Large deviation bounds}

We require some basic upper bounds on the probability that $\abs{\cA \cap [M]}$ is large, for various ranges of $M$.  In applications, $\omega$ will be a parameter that tends slowly to infinity as $p \to 0$.

\begin{lemma}\label{lem:lower-small-values}
For any $\omega>0$, we have
$$\mathbb{P}\left[\abs*{\cA \cap \left[\frac{p^{-1}}{\omega}\right]}\geq 1\right] \leq \frac{1}{\omega}$$
and
$$\mathbb{P}\left[\abs*{\cA \cap \left[p^{-1} \omega\right]} \geq \omega^2 \right] \leq \frac{1}{\omega}.$$
\end{lemma}

\begin{proof}
The first-moment method gives
$$\mathbb{P}\left[\abs*{\cA \cap \left[\frac{p^{-1}}{\omega}\right]}\geq 1\right]  \leq \mathbb{E}\left(\abs*{\cA \cap \left[\frac{p^{-1}}{\omega}\right]}\right) =p \cdot \frac{p^{-1}}{\omega}=\frac{1}{\omega}$$
and
\[\mathbb{P}\left[\abs*{\cA \cap \left[p^{-1} \omega\right]} \geq \omega^2 \right] \leq \frac{1}{\omega^2} \cdot \mathbb{E}\left(\abs*{\cA \cap \left[p^{-1} \omega\right]} \right) =\frac{\omega}{\omega^2}=\frac{1}{\omega}. \qedhere\]
\end{proof}

\begin{lemma}\label{lem:lower-large-values}
Let $\omega \geq 1$ be an integer.  With probability at least $1-O\left( e^{-\omega/3}\right)$, we have
$$\abs{\cA \cap [mp^{-1}]} \leq 2m \quad \text{simultaneously for all integers $m \geq \omega$}.$$
\end{lemma}

\begin{proof}
Note that $\mathbb{E}(\abs{\cA \cap [mp^{-1}]})=m$.  Chernoff's bound with $\beta=2$ gives
$$\mathbb{P}[\abs{\cA \cap [mp^{-1}]} \geq 2m] \leq e^{-(2\log 2-1)m}<e^{-m/3}.$$
By a union bound, the probability that $\abs{\cA \cap [mp^{-1}]} \geq 2m$ for some integer $m \geq \omega$ is at most
\[\sum_{m=\omega}^\infty e^{-m/3} \lesssim e^{-\omega/3}. \qedhere\]
\end{proof}

\subsection{Integer partition asymptotics}

We require the following classical estimate on the number $\Part(n)$ of integer partitions of a natural number $n$.
\begin{theorem}[Hardy--Ramanujan~\cite{HR}; Uspensky~\cite{Uspensky}]\label{thm:partitions}
We have $$\Part(n)=(1+o(1))\frac{1}{4\sqrt{3} \cdot n} \exp(\pi \sqrt{2n/3}).$$
\end{theorem}

\begin{lemma}\label{lem:partition-consequence}
There is a constant $C>0$ such that the following holds. Let $0<\gamma \leq 1$, and suppose that $\cA_0 \subset \ZZ_{>0}$ satisfies $\abs{\cA_0 \cap [n]} \leq \gamma n$ for all $n \in \ZZ_{\geq 0}$.  Then
$$\abs{\gen{\cA_0} \cap [N]} \leq C \frac{1}{\sqrt{\gamma N}}\exp(\pi \sqrt{2\gamma N/3})$$
for all $N \in \NN$.
\end{lemma}

The upper bound on $\abs{\gen{\cA_0}\cap [N]}$ given in this lemma is worse than the trivial bound for large $N$, but it is quite strong in the relevant range.

\begin{proof} Write $\cA_0=\{a_1,a_2,\ldots\}$ with $a_1<a_2<\cdots$.  Notice that $a_{i}\geq \gamma^{-1}i$ for all $i \geq 1$: indeed, if this failed for some $i$, then we would obtain
$$\abs{\cA_0 \cap [n]} \geq i>\gamma n$$
for $n=\lceil \gamma^{-1}i\rceil-1$, contradicting our density hypothesis. We now connect this property to the asymptotics of $\Part(n)$. Notice that $\abs{\gen{\cA_0} \cap [N]}$ is at most the number of tuples $(b_1, b_2, b_3, \ldots)$ of nonnegative integers with
$$a_1b_1+a_2 b_2+a_3 b_3+\cdots \leq N.$$
Our lower bound on the $a_i$'s guarantees that every such tuple also satisfies
$$b_1+2b_2+3b_3+\cdots \leq \gamma N.$$
Thus, using \cref{thm:partitions}, we conclude that the total number of tuples $(b_1, b_2, b_3, \ldots)$ is at most
\begin{align*}
\sum_{n=0}^{\lfloor \gamma N \rfloor} \Part(n) &= 1+\sum_{n=1}^{\lfloor \gamma N \rfloor} (1+o(1))\frac{1}{4\sqrt{3} \cdot n} \exp(\pi \sqrt{2n/3})\\
&\lesssim \frac{1}{\sqrt{\gamma N}}\exp(\pi \sqrt{2\gamma N/3}),
\end{align*}
as desired.
\end{proof}

\subsection{Sparseness of \texorpdfstring{$\gen{\cA}$}{<A>}}

The following proposition makes precise the assertion that $\gen{\cA}$ is very sparse up to scale $\asymp p^{-1}(\log p^{-1})^2$.

\begin{proposition}\label{prop:sparse}
For any $\varepsilon>0$, there is some $c=c(\varepsilon)>0$ such that with high probability we have $\abs{\gen{\cA} \cap [N]} \leq N^{\varepsilon}$, where $N:=cp^{-1} (\log p^{-1})^2$.
\end{proposition}

\begin{proof}
Let $\omega \geq 1$ be an integer parameter (depending on $p$) to be determined later.  Set $$\cA_1:=\cA \cap \left[\frac{p^{-1}}{\omega}\right], \quad \cA_2:=\cA \cap \left(\frac{p^{-1}}{\omega}, p^{-1}\omega\right], \quad \text{and} \quad \cA_3:=\cA \cap \left(p^{-1}\omega,N\right],$$
so that $\cA=\cA_1 \cup \cA_2 \cup \cA_3$. \cref{lem:lower-small-values,lem:lower-large-values} together with a union bound tell us that with probability at least
$$1-\frac{1}{\omega}-\frac{1}{\omega}-O\left( e^{-\omega/3}\right)=1-O\left(\frac{1}{\omega}\right),$$
we have
$$\abs{\cA_1}=0, \quad \abs{\cA_2} \leq \omega^2, \quad \text{and} \quad \abs{\cA_3 \cap [mp^{-1}]} \leq 2m \text{ for all integers $m \geq \omega$}.$$
Suppose we are in such an outcome.

Since $\abs{\gen{\cA} \cap [N]} \leq \abs{\gen{\cA_2} \cap [N]} \cdot \abs{\gen{\cA_3} \cap [N]}$, it will suffice to show that the latter product is at most $N^{\varepsilon}$.  To this end, notice that each element of $\gen{\cA_2} \cap [N]$ can be expressed as a sum of at most $\omega pN$ elements of $\cA_2$, whence (crudely)
 $$\abs{\gen{\cA_2} \cap [N]} \leq \left( \omega pN +1\right)^{\omega^2}\leq\left(2c\omega (\log p^{-1})^2 \right)^{\omega^2}=\exp\left(\omega^2 \left( \log(2c \omega)+2\log\log p^{-1} \right) \right).$$
Next, we use \cref{lem:partition-consequence}. Our condition on $\cA_3$ implies, for each $n\geq p^{-1}\omega$, that
\[\abs{\cA_3\cap [n]}\leq 2\lceil np\rceil\leq 3pn\]
as long as, say, $\omega\geq 2$. The same condition holds for $n<p^{-1}\omega$ since $\cA_3\cap [0,p^{-1}\omega)$ is empty. So, applying \cref{lem:partition-consequence} with $\gamma:=3p$, we have
$$\abs{\gen{\cA_3} \cap [N]} \leq \exp\left(C\sqrt{pN}\right)=\exp \left( C\sqrt{c} \cdot \log p^{-1}  \right)
$$
for some absolute constant $C>0$.  It follows that
\begin{align*}
\abs{\gen{\cA} \cap [N]} & \leq \abs{\gen{\cA_2} \cap [N]} \cdot \abs{\gen{\cA_3} \cap [N]}\\
 &\leq \exp\left(\omega^2 \left( \log(2c \omega)+2\log\log p^{-1} \right)+C\sqrt{c} \cdot \log p^{-1} \right).
\end{align*}
If $c$ is sufficiently small (e.g., $\sqrt{c}=\varepsilon/(2C)$) and $\omega$ is growing sufficiently slowly as $p \to 0$ (e.g., $\omega=\log \log p^{-1}$), then this last quantity is indeed at most $N^{\varepsilon}$ for $p$ sufficiently small.
\end{proof}

Taking any $\varepsilon<1$, this proposition immediately implies that with high probability $g(\gen{\cA})\gtrsim p^{-1} (\log p^{-1})^2$.  It remains to establish the lower bound on $e(\gen{\cA})$.

\begin{proposition}
With high probability we have $e(\gen{\cA}) \gtrsim (\log p^{-1})^2$.
\end{proposition}

\begin{proof}
Set $N:=cp^{-1} (\log p^{-1})^2$ for a small constant $c>0$ to be determined later.  Define the sets
$$\cB:=\gen{\cA \cap [N/2]} \cap [N] \quad \text{and} \quad \cA':=\cA \cap (N/2,N],$$
and notice that $\cB, \cA'$ are independent random subsets of $[N]$.
Recall that $e(\gen{\cA})$ is the number of elements of $\cA$ that cannot be expressed as sums of other elements of $\cA$.  We will show that, with high probability, $\cA'$ is large (of size $\gtrsim (\log p^{-1})^2$) and almost all of the elements of $\cA'$ contribute to $e(\gen{\cA})$.

Consider a fixed outcome $\cB_0$ of $\cB$ (with $\cA'$ still random).  Since the sum of any two elements of $\cA'$ is strictly larger than $N$, an element of $\cA'$ contributes to $e(\gen{\cA})$ if and only if it is not contained in $\cB_0$ or the sumset $\cB_0+\cA'$.  Set $$X:=\abs{\cB_0 \cap \cA'} \quad \text{and} \quad Y:=\abs{\cA' \cap (\cB_0+\cA')}.$$
Then at least $\abs{\cA'}-X-Y$ elements of $\cA'$ contribute to $e(\gen{\cA})$.  We will show that $X+Y$ is very small in expectation.  First, we have $\EE(X)\leq p\abs{\cB_0}$.  Second, for each $b \in \cB_0$, there are at most $N/2$ pairs $(a_1,a_2) \in (N/2,N] \times (N/2,N]$ with $a_1-a_2=b$, so the expected number of pairs $(a_1,a_2) \in \cA' \times \cA'$ with $a_1-a_2=b$ is at most $p^2N/2$.  Hence the expected number of triples $(a_1,a_2,b) \in \cA' \times \cA' \times \cB_0$ with $a_1-a_2=b$ is at most $\abs{\cB_0}p^2 N/2$, and this is an upper bound for $\EE(Y)$.  So
$$\EE(X+Y) \leq p\abs{\cB_0}+\abs{\cB_0}p^2 N/2=p\abs{\cB_0}(1+pN/2).$$
Markov's Inequality tells us that with probability at least $1-\sqrt{p}=1-o(1)$, we have $X+Y \leq \sqrt{p} \cdot\abs{\cB_0}(1+pN/2)$ and hence $$e(\gen{\cA}) \geq \abs{\cA'}-\sqrt{p} \cdot\abs{\cB_0}(1+pN/2).$$

To conclude, it suffices to show that with high probability $\abs{\cA'}-\sqrt{p} \cdot \abs{\cB}(1+pN/2) \gtrsim (\log p^{-1})^2$.  We have $\EE(\abs{\cA'})=pN/2$, and Chernoff's bound with $\alpha=1/2$ tells us that with high probability
$$\abs{\cA'} \geq pN/4=c(\log p^{-1})^2/4.$$
At the same time, if $c$ is sufficiently small, then \cref{prop:sparse} with $\varepsilon=1/3$ tells us that with high probability $\abs{\cB} \leq \abs{\gen{\cA}} \leq N^{1/3}$.  In particular, in such an outcome we have
$$\sqrt{p} \cdot \abs{\cB}(1+pN/2)\leq \sqrt{p} \cdot N^{1/3}(1+c(\log p^{-1})^2)=o(1),$$
which is certainly acceptable.
\end{proof}

\section{Upper bounds}\label{sec:upper}
We now turn to the upper bounds in \cref{thm:main}.  The main work is proving the with-high-probability statements.  Once we have these statements (with sufficiently good quantitative control on the $o(1)$ failure probabilities), the corresponding expectation statements will follow from crude upper bounds on $F(\gen{\cA})$ and the deterministic inequalities $g(\cS) \leq F(\cS)+1$ and $e(\cS) \leq F(\cS)+2$.

Our strategy is similar to the strategy from~\cite{BM}.  Let $q:=\min(\cA)$.  The key observation is that if every residue class modulo $q$ is represented in $\gen{\cA} \cap [N]$, then $\gen{\cA \cap [N]}$ contains all integers greater than or equal to $N$.  It follows that $F(\gen{\cA}),\ g(\gen{\cA})< N$ and $e(\gen{\cA}) \leq \abs{\cA \cap [N]}$. 

With high probability $q$ is not much bigger than $p^{-1}$.  We consider the sequences of images modulo $q$ of the elements of $\cA \cap [M]$, for some $M \gtrsim p^{-1} \log q$; with high probability there are $\gtrsim \log q$ such elements.  The core of our argument uses the second moment method to show that if $x_1, \ldots, x_L$ is a uniformly random sequence of elements of $\ZZ/q\ZZ$ with $L \geq 10\log q$ (say), then
\[\big\{\eps_1x_1+\cdots+\eps_L x_L:(\eps_1,\ldots,\eps_L)\in\{0,1\}^L\big\}=\ZZ/q\ZZ\]
with very high probability.  Of course, our sequence of residues coming from $\cA \cap [M]$ is not a uniformly random sequence, but we can obtain a uniformly random sequence from it via subsampling. (A subsampling procedure, encapsulated in \cref{lem:distinct-resample}, is the key technical input in this paper which enables us to save a factor of $\log p^{-1}$ over the result of \cite{BM}. In \cite{BM}, $q$ is chosen to be prime to gain additional symmetry in $\ZZ/q\ZZ$, while we are able to treat arbitrary $q$.) At this point we conclude that every residue class modulo $q$ is occupied by some sum of at most $L \lesssim \log q$ elements of $\cA \cap [M]$; such a sum has size at most $\lesssim M \log q \lesssim p^{-1} (\log p^{-1})^2$.  The following subsections contain the details.

\subsection*{Covering all residues}
We start with the main second moment method calculation.

\begin{lemma}\label{lem:coverall} Let $q,L \in \mathbb{Z}_{>0}$ with $2^L \geq q$. If $x_1,\ldots,x_L$ are chosen independently and uniformly at random from $\ZZ/q\ZZ$, then 
\[\big\{\eps_1x_1+\cdots+\eps_L x_L:\eps_1,\ldots,\eps_L\in\{0,1\}^L\big\}=\ZZ/q\ZZ\]
with probability at least $1 -q^2 \cdot 2^{-L}$.
\end{lemma}

\begin{proof}
    For each nonzero $t \in \mathbb{Z}/q\mathbb{Z}$, let $X_t$ denote the number of vectors $\vec{\varepsilon} \in \{0,1\}^L \setminus \{\vec{0}\}$ such that
\[
\sum_{i=1}^L \varepsilon_i x_i = t.
\]
We will show that each $\Pr[X_t = 0]$ is small and then use a union bound over $t$.
For each nonzero $\vec{\varepsilon} \in \{0,1\}^L$, since the $x_i$'s are chosen uniformly and independently, the dot product $\vec{\varepsilon} \cdot \vec{x}$ is uniformly distributed on $\mathbb{Z}/q\mathbb{Z}$. Thus
\[
\mathbb{E}[X_t] = \sum_{\vec{\varepsilon} \ne \vec{0}} \Pr[\vec{\varepsilon} \cdot \vec{x} = t] = (2^L - 1) \cdot \frac{1}{q} =: \mu.
\]
To compute the second moment, we expand
\[
\mathbb{E}[X_t^2] = \sum_{\vec{\varepsilon}, \vec{\delta} \ne \vec{0}} \Pr[\vec{\varepsilon} \cdot \vec{x} = \vec{\delta} \cdot \vec{x} = t].
\]
The contribution from each diagonal term $\vec{\varepsilon} = \vec{\delta}$ is $1/q$. If $\vec{\varepsilon} \ne \vec{\delta}$, then the dot products $\vec{\varepsilon}\cdot\vec{x}$ and $\vec{\delta}\cdot\vec{x}$ are independent uniformly random elements of $\ZZ/q\ZZ$, and the probability of both equaling $t$ is exactly $1/q^2$. Thus
\[
\mathbb{E}[X_t^2]= (2^L-1)\cdot \frac{1}{q}+(2^L-1)(2^L-2) \cdot \frac{1}{q^2},
\]
and the variance of $X_t$ is
\begin{align*}
\operatorname{Var}(X_t) &= \mathbb{E}[X_t^2] - \mathbb{E}[X_t]^2\\
 &=(2^L-1)\cdot \frac{1}{q}+(2^L-1)(2^L-2) \cdot \frac{1}{q^2}-(2^L-1)^2 \cdot \frac{1}{q^2}\\
 &=(2^L-1)\cdot \left(\frac{1}{q}-\frac{1}{q^2} \right)\\
 &=\mu\left(1 - \frac{1}{q}\right).
\end{align*}
Chebyshev's inequality gives
\[
\Pr[X_t = 0] \le \Pr\left[\abs{X_t - \mu} \ge \mu\right] \le \frac{\operatorname{Var}(X_t)}{\mu^2} = \frac{1 - \frac{1}{q}}{\mu} < \frac{1}{\mu}.
\]
So, for each fixed $t \neq 0$, the probability that $t$ is not expressible as $t=\eps_1x_1+\cdots+\eps_L x_L$ is at most $1/\mu$.  Union-bounding over all $t$ (and noting that $0$ is always expressible), we conclude that
\[\Pr[\{\eps_1x_1+\cdots+\eps_L x_L:\eps_1,\ldots,\eps_L\in\{0,1\}^L\big\} \neq \ZZ/q\ZZ]< \frac{q-1}{\mu}\leq \frac{q^2}{2^{L}},\]
where the final inequality used the assumption $2^L \geq q$.
\end{proof}

We remark that the proof establishes the same result with $\ZZ/q\ZZ$ replaced by any other abelian group of order $q$.

\subsection*{Sampling to uniform}

We now describe the subsampling procedure which will allow us to obtain a uniformly random sequence of elements of $\ZZ/q\ZZ$ from our random set $\cA \cap [M]$. 

\begin{lemma}\label{lem:distinct-resample} Let $q,\ell \in \ZZ_{>0}$ with $\ell\leq q$, and let $y_1,\ldots,y_\ell$ be uniformly random elements of $\ZZ/q\ZZ$, conditioned to be distinct. Choose $z_1,\ldots,z_\ell$ independently and uniformly at random from $\ZZ/q\ZZ$. Now choose $x_1,\ldots,x_\ell$ uniformly at random from $\{y_1,\ldots,y_\ell\}$ subject to the condition that $x_i=x_j$ if and only if $z_i=z_j$.  Then the distribution of $x_1,\ldots,x_\ell$ is that of $\ell$ independent uniformly random elements of $\ZZ/q\ZZ$.
\end{lemma}

\begin{proof}
We will correspond length-$\ell$ sequences of elements of $\ZZ/q\ZZ$ to pairs $(\cP,s)$ where $\cP$ is a partition of $[\ell]$ and $s$ is a sequence of $\abs{\cP}$ (the number of parts of $\cP$) \emph{distinct} elements of $\ZZ/q\ZZ$. Given a sequence of elements $w_1,\ldots,w_\ell\in\ZZ/q\ZZ$, let $\cP(w_1,\ldots,w_\ell)$ be the partition where $i,j$ are in the same part if and only if $w_i=w_j$, and let $s(w_1,\ldots,w_\ell)$ be the sequence of distinct elements appearing in $w_1, \ldots, w_\ell$.  If $w_1, \ldots, w_\ell$ is a uniformly random sequence, then for each partition $\cP$, the distribution of $s(w_1, \ldots, w_\ell)$ is uniform on the length-$\abs{\cP}$ sequences of distinct elements after we condition on $\cP(w_1, \ldots, w_\ell)=\cP$.

Consider now our random sequence $x_1, \ldots, x_\ell$.  Since $\cP(x_1, \ldots, x_\ell)=\cP(z_1, \ldots, z_\ell)$, the distribution of $\cP(x_1, \ldots, x_\ell)$ is the same as the distribution coming from a uniformly random sequence.  Fix a partition $\cP$, and condition on the event $\cP(x_1, \ldots, x_\ell)=\cP$; this depends only on $z_1, \ldots, z_\ell$.  Then the sequence $s(x_1, \ldots, x_\ell)$ is uniformly distributed on the length-$\abs{\cP}$ sequences of distinct elements because $y_1, \ldots, y_\ell$ is unaffected by the conditioning.  Thus the distribution of $x_1, \ldots, x_\ell$ is the same as the distribution of a uniformly random sequence.
\end{proof}

\subsection*{Hitting all large values}
It is time to combine the pieces and show that with very high probability $\gen{\cA}$ contains all integers past $\asymp p^{-1} (\log p^{-1})^2$.

\begin{proposition}\label{prop:dense}
For any $K \geq 1$, there is some $C=C(K)>0$ such that with probability at least $1-O(p^K)$, the numerical semigroup $\gen{\cA}$ contains all integers greater than or equal to $N:=Cp^{-1} (\log p^{-1})^2$, and we have $\gen{\cA \cap [N]}=\gen{\cA}$ and $\abs{\cA \cap [N]} \leq 2C (\log p^{-1})^2$.
\end{proposition}

\begin{proof}
Let $q:= \min (\cA)$.  The probability that $q >Kp^{-1} \log p^{-1}$ is
$$(1-p)^{Kp^{-1} \log p^{-1}} \leq e^{-K\log p^{-1}}=p^K.$$
Suppose we are in an outcome where $q\leq Kp^{-1} \log p^{-1}$.  Set
$$\cA':=\cA \cap [q+1,(m+1)q], \quad \text{where} \quad  m:=\lceil 20 p^{-1}q^{-1} \log q+20Kp^{-1}q^{-1}\log p^{-1}\rceil.$$
We have $\EE(\abs{\cA'})=pmq$.  Since $$20 \log q+20K\log p^{-1} \leq pmq \leq 22 \log q+22K\log p^{-1},$$
Chernoff's bound with $\alpha=1/2$ and $\beta=3/2$ tells us that $$10\log q+10K \log p^{-1} \leq \abs{\cA'} \leq 33\log q+33K \log p^{-1}$$
with probability at least 
$$1-e^{-\frac18\EE\abs{\cA'}}-e^{-(0.108)\EE\abs{\cA'}}>1-2e^{(20K \log p^{-1})/10}\geq 1-p^{K}.$$
Fix some $L \geq 10\log q+10K \log p^{-1}$, and suppose that we are in an outcome with $\abs{\cA'}=L$.

For each $1 \leq i \leq m$, set
$$\cA_i:=\cA \cap [iq+1, (i+1)q],$$
and notice that $\abs{\cA_1}+\cdots+\abs{\cA_{m}}=\abs{\cA'}=L$.  Applying \cref{lem:distinct-resample} to (the image modulo $q$ of) each $\cA_i$ and concatenating the resulting sequences, we obtain a sequence $x_1, \ldots, x_L \in \cA'$ whose reduction to $(\ZZ/q\ZZ)^L$ is uniformly random. Notice that in each application of \cref{lem:distinct-resample}, we have introduced further randomness, both in the random sequence $z_1,\ldots,z_{\abs{\cA_i}}$ and in the way in which $x_1,\ldots,x_{\abs{\cA_i}}$ are selected from $\cA_i$. For each $1\leq i\leq m$, encapsulate this data in a random variable $w^{(i)}$.  Now \cref{lem:coverall} tells us that we have
\begin{equation}\label{eq:cover}
\big\{\eps_1x_1+\cdots+\eps_L x_L:(\eps_1,\ldots,\eps_L)\in\{0,1\}^L\big\} \pmod{q}=\ZZ/q\ZZ
\end{equation}
with probability at least $$1-q^2 \cdot 2^{-L} \geq 1-p^K$$
(with plenty of room to spare).  In particular, with probability at least $1-p^K$, the random set $\mathcal{A}$ is such that \eqref{eq:cover} holds for \emph{some} choice of the data $w^{(1)}, \ldots, w^{(m)}$.  Suppose we are in such an outcome, and fix $w^{(1)}, \ldots, w^{(m)}$ accordingly.  Each sum $\eps_1x_1+\cdots+\eps_L x_L$ has size at most
$$(m+1)q \cdot L \lesssim Kp^{-1} (\log p^{-1}) \cdot K \log p^{-1}=K^2 p^{-1} ( \log p^{-1})^2.$$
It follows that $\cA$ contains all integers greater than or equal to $N=Cp^{-1} (\log p^{-1})^2$ as long as $C$ is sufficiently large relative to $K$.  In total, this event occurs with probability at least $1-3p^K$, as desired. 
Finally, Chernoff's bound with $\beta=2$ tells us that $\abs{\cA \cap [N]}\leq 2CpN=2C (\log p^{-1})^2$ with probability at least $1-e^{-C (\log p^{-1})^2/3} \geq 1-p^K$.  By a union bound, with probability at least $1-O(p^K)$, all of these events occur.
\end{proof}

This proposition with $K=1$ (for instance) immediately implies that with high probability $g(\gen{\cA}), F(\gen{\cA}) \lesssim p^{-1}(\log p^{-1})^2$ and $e(\gen{\cA}) \lesssim (\log p^{-1})^2$ (the latter because $\cA \cap [N]$ is a generating set of size $\lesssim (\log p^{-1})^2$).

\subsection*{Tail bounds}
In order to establish the desired expectation statements, we need to show that $F(\gen{\cA})$ is seldom extremely large.  We mostly follow the argument from~\cite[Lemma 4.1]{BM}, which we reproduce for the reader's convenience.

\begin{lemma}\label{lem:large-deviations}
Let $u \geq 2$ be an even number.  We have $\EE(F(\gen{\cA \cap [u,\infty)})) \lesssim p^{-4}+u^2$.
\end{lemma}

\begin{proof}
Let $\cA':=\cA \cap [u,\infty)$.  Let $v$ be the smallest integer such that $2v,2v+1 \in \cA'$; such $v$ exists with probability $1$.  We can crudely bound $$F(\cA') \leq 2v(2v+1)-2v-(2v+1) \lesssim v^2.$$  For each value $v_0 \geq u/2$, we have $\mathbb{P}[v=v_0]=(1-p^2)^{v_0-u/2} \cdot p^2$.  Conditioning on the value of $v$ and summing the resulting infinite series, we find that
\begin{align*}
\EE(F(\gen{\cA'})) & \lesssim \sum_{v=u/2}^{\infty} (1-p^2)^{v-u/2} \cdot p^2 \cdot v^2\\
 &=2p^{-4}+(u-3)p^{-2}+\frac14(u-2)^2\lesssim p^{-4}+u^2.\qedhere
\end{align*}
\end{proof}

\begin{proof}[{Proof of \cref{thm:main}}]
Now apply \cref{prop:dense} with $K=5$, and let $C$ be the resulting constant; set $N:=C p^{-1} (\log p^{-1})^2$.  Decompose $\cA=\cA^* \cup \cA'$, where $\cA^*:=\cA \cap [N]$ and $\cA':=\cA \cap [N+1,\infty)$; without loss of generality we may assume that $N+1$ is even.  We condition on whether or not $\cA^*$ satisfies the conclusion of \cref{prop:dense}.  If it does, which occurs with probability at least $1-O(p^5)$, then we have $$g(\gen{\cA}),F(\gen{\cA}) \leq C p^{-1} (\log p^{-1})^2 \quad \text{and \quad }e(\gen{\cA}) \leq 2C(\log p^{-1})^2.$$
Otherwise, applying \cref{lem:large-deviations} with $u:=N+1$ still gives the bound
$$\EE(F(\gen{\cA}))\leq \EE(F(\gen{\cA'})) \lesssim p^{-4}+N^2 \lesssim p^{-4}.$$
In total (using the deterministic inequalities $g(\cS) \leq F(\cS)+1$ and $e(\cS) \leq F(\cS)+2$) we have
$$\EE(g(\gen{\cA})) \lesssim \EE(F(\gen{\cA})) \lesssim p^{-1} (\log p^{-1})^2+p^5 \cdot p^{-4} \lesssim p^{-1} (\log p^{-1})^2$$
and
$$\EE(e(\gen{\cA})) \lesssim (\log p^{-1})^2 +p^5 \cdot  p^{-4} \lesssim (\log p^{-1})^2,$$
as desired.
\end{proof}

\section{Further questions}\label{sec:concluding}
We conclude with a few suggestions for future work.
\begin{enumerate}
    \item Our results establish the with-high-probability orders of magnitude of $F(\gen{\cA}),g(\gen{\cA}),e(\gen{\cA})$ when $\cA \subseteq \ZZ_{>0}$ is density-$p$ random.  It would be interesting to say something more precise about the distributions of these parameters after appropriate centering and rescaling.  At this point, it is not even obvious what scale of concentration to expect (e.g., does $F(\gen{\cA})$ have fluctuations of order $\asymp p^{-1} (\log p^{-1})^2$ or of smaller order?).

    \item In a similar vein, we showed that typically $\gen{\cA}$ is very sparse below a small constant times $p^{-1} (\log p^{-1})^2$ and contains all integers past a large constant times $p^{-1} (\log p^{-1})^2$.  What happens in the intermediate regime, and how ``sharp'' is the transition from sparseness to denseness?
    
    \item Of course, one could ask about the typical behavior of other numerical semigroup invariants, such as the catenary degree, the tame degree, and the size of the $\Delta$-set.

    \item Finally, one could study analogous problems for random generalized numerical semigroups (co-finite subsets of $\ZZ_{\geq 0}^d$ that contain $0$ and are closed under addition).\footnote{Note added in revision: This problem has been treated in follow-up work by Bitonti and the first author~\cite{BS}.}
    
\end{enumerate}

\section*{Acknowledgements}
Much of this research was conducted while the authors were attending the 2025 summer program on probabilistic and extremal combinatorics at the Park City Mathematics Institute.  We thank PCMI and the IAS for providing an outstanding working environment. Morales is supported by the National Science Foundation under grants DMS-2348578 and ~DMS-2434665. Schildkraut is supported by the National Science Foundation Graduate Research Fellowship Program under Grant No.~DGE-2146755.


\begin{thebibliography}{99}

\bibitem{Aliev2011}
I. Aliev, M. Henk, and A. Hinrichs, Expected Frobenius numbers.
\emph{J. Combin. Theory Ser. A}, {\bf 118}(2) (2011), 525--531.

\bibitem{AlonSpencer}
N. Alon and J. H. Spencer,
\emph{The Probabilistic Method},
3rd ed., Wiley, 2008.

\bibitem{Arnold1999}
V. I. Arnold, Weak asymptotics for the numbers of solutions of Diophantine problems.
\emph{Functional Analysis and Its Applications}, {\bf 33}(4) (1999), 292--293.

\bibitem{Assi2020}
A. Assi, M. D’Anna, and P. A. García-Sánchez, Numerical Semigroups and Applications.
Springer Nature, vol.~3, 2020.

\bibitem{BS} V. Bitonti and N. Kravitz, Gap sets of random generalized numerical semigroups.  \emph{Preprint} arXiv:2604.24891v1 (2026).

\bibitem{BM} T. Bogart and S. Morales, Improved Upper Bounds on Key Invariants of Erd\H{o}s-R\'enyi Numerical Semigroups.  \emph{Preprint} arXiv:2411.13767v3 (2025).

\bibitem{BKS} M. Bras-Amor\'{o}s, N. Kaplan, and D. Singhal, The shape of a random numerical semigroup.  \emph{Preprint} arXiv:2604.26127v1 (2026).

\bibitem{Delgado2020}
M. Delgado, Conjecture of Wilf: a survey.
\emph{Numerical Semigroups: IMNS 2018} (2020), 39--62.

\bibitem{DLOW} J. De Loera, C. O'Neill, and D. Wilburne, Random numerical semigroups and a simplicial complex of irreducible semigroups.  \emph{Electr. J. Combin.}, {\bf 25(4)} (2018), \#P4.37.

\bibitem{HR} G. H. Hardy and S. Ramanujan, Asymptotic Formulae in Combinatory Analysis.  \emph{Proc. London Math. Soc.}, {\bf 17} (1918), 75--115.

\bibitem{Rosales2009}
J. C. Rosales and P. A.García-Sánchez, Numerical Semigroups. Springer, 2009.

\bibitem{Uspensky} J. V. Uspensky, Asymptotic Formulae for Numerical Functions Which Occur in the Theory of Partitions.  \emph{Bull. Acad. Sci. Russie}, {\bf 14} (1920), 199--218.


\end{thebibliography}
\end{document}